\documentclass{amsart}
\usepackage{amsthm, amssymb, array, enumerate, mathrsfs, bbm,  times, txfonts, verbatim}
\usepackage[initials]{amsrefs}
\pdfoutput=1
\let\rsf\mathscr

\def\dblCaret{\mathbin{\hbox{$\hat{\hphantom{m}}$}\kern-10pt\hbox{$\hat{\hphantom{m}}$}}}
\def\eqcl[#1]{\pmb{[}#1\pmb{]}}
\def\one{\mathbf{1}}
\def\zero{\mathbf{0}}
\def\env{\operatorname{env}}
\def\cork{\operatorname{corank}}
\newcommand{\card}[1]{\left| #1\right|}

\title{Finite Implication Algebras}
\author{Colin G. Bailey}
\address{School of Mathematics, Statistics \& Operations Research\\
Victoria University of Wellington\\
Wellington, New Zealand\\
}
\email{Colin.Bailey@vuw.ac.nz}
\author{Joseph S.Oliveira}
\address{
Pacific Northwest National Laboratories\\
Richland, WA\\
U.S.A.}
\email{Joseph.Oliveira@pnl.gov}
\date{\today}
\providecommand{\meet}{\mathbin{\wedge}}
\providecommand{\join}{\mathbin{\vee}}
\newcommand{\comp}[1]{\overline{#1}}
\ifx\upharpoonright\undefined
     \def\restrict{\hbox{\rm\kern0.166em\accent"12\kern-0.536em$\vert$\kern0.3em}}%
  \else
     \def\restrict{\upharpoonright}%
  \fi
\makeatletter
\def\twoSet#1#2{\left\{%
\vphantom{#2}#1\thinspace\right|\nolinebreak[3]\left.%
  #2%
  \vphantom{#1}%
  \right\}%
}
\def\oneSet#1{\left\lbrace#1\right\rbrace}

\newif\if@nstr
\def\setstrfalse{\let\if@nstr=\iffalse}
\def\setstrtrue{\let\if@nstr=\iftrue}
\def\@nstr #1#2{
\def\@@nstr ##1#1##2##3\@@nstr{\ifx
\@nstr ##2\setstrfalse \else \setstrtrue \fi }
\@@nstr #2#1\@nstr \@@nstr}
\def\@separate#1|#2@{\setFront{#1}\setBack{#2}}
\def\lb#1\rb{\@nstr|{#1} \if@nstr \@separate#1 @ \twoSet{\@setFront}{\@setBack}%
\else \@separate |{#1 }@ \oneSet{\@setBack}\fi%
}
\def\setFront#1{\def\@setFront{#1}}
\def\setBack#1{\def\@setBack{#1}}
\def\Set#1{\lb{#1}\rb}

\def\oneBrk#1{\left\langle#1\right\rangle}
\def\twoBrk#1#2{\left\langle%
\vphantom{#2}#1\thinspace\right|\nolinebreak[3]\left.%
  #2%
  \vphantom{#1}%
  \right\rangle%
}
\def\brk<#1>{\@nstr|{#1} \if@nstr \@separate#1 @ \twoBrk{\@setFront}{\@setBack}%
\else \@separate |{#1 }@ \oneBrk{\@setBack}\fi%
}

\def\thmref#1{\normalfont{theorem}~\ref{#1}}
\def\lemref#1{\normalfont{lemma}~\ref{#1}}

\theoremstyle{plain}
\newtheorem{thm}{Theorem}[section]
\newtheorem{lem}[thm]{Lemma}

\newtheorem{defn}[thm]{Definition}

\theoremstyle{remark}
{}
{}
{}
{}

\@ifpackageloaded{bbm}{

\newcommand{\N}{{\mathbbm{N}}}

}{

\newcommand{\N}{{\mathbb{N}}}

}
\makeatother

\begin{document}

\begin{abstract}
	We consider several distinct characterizations of finite 
	implication algebras. One of these leads to a new 
	characterization of Boolean polymatroids.
\end{abstract}
\maketitle

\section{Introduction}
Implication algebras were introduced by J.C.Abbott (\cite{Abb}) as a  
way of 
considering semi-Boolean algebras with a single operation. Finite 
Boolean algebras are easily described and well understood,  but similar 
descriptions for implication algebras are lacking. In this work we 
show that finite implication algebras are cryptomorphic to hypergraphs.
Then we give a construction of the associated graphical 
polymatroid from the implication algebra viewpoint and deduce certain 
algebraic properties. Finally we show that these properties characterise all 
such polymatroids. A related characterisation can be found in 
\cite{Whi}.

\section{Finite Implication Algebras}
We know that the isomorphism type of a finite Boolean is completely determined 
by the number of atoms. Of course this cannot be true for finite 
implication algebras, but there are a fairly simple set of 
invariants  that do determine the isomorphism type. 

\begin{defn}\label{def:minRk}
    Let $\mathcal{I}$ be a finite implication algebra. Let 
    \begin{align*}
	M_I & =\Set{a | a\in\mathcal{I}\text{ is minimal}}  \\
	n(I) & =\card{M_I}\\
	\intertext{$r\colon \wp(M_I)\setminus\Set{\emptyset}\to 
	\N$ is defined by}
	p_{I}(S)=k\text{ iff }[\bigvee S,  1]\simeq 2^{k}.
    \end{align*}
    The function $p_{I}$ ic called 
    the \emph{implication profile} of $\mathcal{I}$.
\end{defn}

It is easy to see that $p_{I}(S)$ is the height of $[\bigvee S,  1]$.

\begin{thm}\label{thm:minRk}
    Let $\mathcal{I}_{1}$ and $\mathcal{I}_{2}$ be two finite 
    implication algebras. Then 
    \begin{align*}
    \mathcal{I}_{1}\simeq\mathcal{I}_{2} \text{ iff } & 
    n(I_{1})=n(I_{2})\text{ and }\\
    &\qquad \text{ there is a bijection }\varphi\colon M_{I_{1}}\to 
    M_{I_2}\text{ such that 
    }\\
    &\qquad p_{I_{1}}(X)=p_{I_{2}}(\varphi[X])\text{ for all 
    }X\subseteq M_{I_{1}}.
    \end{align*}
\end{thm}
\begin{proof}
    The left to right direction is clear. 
    
    Suppose that we have 
    $n= n({I}_{1})= n({I}_{2})$ and 
    a bijection $\varphi\colon M_{I_1}\to 
	M_{I_2}$ such that  $p_{I_1}(X)=p_{I_2}(\varphi[X])$ for 
	all $X\subseteq M_{I_1}$.   
	
   We define by backwards induction a family of isomorphisms
   $f_{X}\colon [\bigvee X, \one]\to [\bigvee \varphi[X], \one]$ such 
   that $X\subseteq Y$ implies $f_{Y}\subseteq f_{X}$. The desired 
   mapping is then $\bigcup_{a\in M_{I_1}}f_{\Set a}$. 
   
   The base case is $X=M_{I_1}$ and we know that 
   $\varphi[M_{I_1}]=M_{I_2}$ and
   $p_{I_1}(\bigvee M_{I_1})= p_{I_2}(\bigvee M_{I_2})$ 
   so that 
   $[\bigvee M_{I_1}, \one]\simeq [\bigvee M_{I_2}, \one]$. 
   Let $f_{\mathcal{I}_{1}}$ be any Boolean isomorphism between 
   these two Boolean algebras. 
   
   Now suppose that $\emptyset\not= X\subseteq M_{I_1}$ 
   and $\card X=k$. Then we know that for all $a\in 
   M_{I_1}\setminus X$ we have a mapping $f_{X\cup\Set a}$ 
   and that $a\not=b$ implies 
   $f_{X\cup\Set a}\restrict[\bigvee X\join a\join b]= 
   f_{X\cup\Set b}\restrict[\bigvee X\join a\join b]= f_{X\cup\Set{a, b}}$. 
   
   Thus we may glue these mappings together to obtain a mapping $f'$ from 
   $[\bigwedge_{a\in M_{I_1}\setminus X}\left(\bigvee X\join 
   a\right), \one]$ to\\
   $[\bigwedge_{a\in M_{I_2}\setminus \varphi[X]}\left(\bigvee \varphi[X]\join 
   \varphi(a)\right), \one]$. $f'$ is an isomorphism as each of its 
   components are isomorphisms. 
   
   Note that $\bigwedge_{a\in M_{I_1}\setminus X}\left(\bigvee X\join 
      a\right)=\bigvee X\join(\bigwedge_{a\in M_{I_1}\setminus X}a)$
   
   Now we have $p_{I_1}(X)= p_{I_2}(\varphi[X])$ so that 
   $[\bigvee X, \one]\simeq [\bigvee\varphi[X], \one]$ and hence \\
   $[(\bigwedge_{a\in M_{I_1}\setminus X}a)\to \bigvee X, \one]\simeq 
   [(\bigwedge_{a\in M_{I_2}\setminus \varphi[X]}a)\to\bigvee\varphi[X], \one]$. 
   Let $f''$ be any isomorphism between these latter two intervals. If 
   $f_{X}$ is the result of the 
   natural gluing of $f'$ and $f''$ then everything works. 
\end{proof}

The number of coatoms of $\mathcal I$ is also important as it 
determines the \emph{enveloping algebra} of $\mathcal I$ -- see 
below, \lemref{lem:envCoat}. The enveloping algebra of an implication 
algebra is the minimal Boolean algebra in which $I$ embeds as an 
upwards-closed sub-implication algebra. These exist for all 
implication algebras. 

\section{Hypergraphs as Implication algebras}

\begin{defn}\label{def:hypergraph}
    A \emph{hypergraph} is a pair $\rsf{H}=\brk<V, E>$ where $H$ is a 
    finite set of \emph{vertices} and $E\subseteq H$ is a set of \emph{edges}. 
\end{defn}

We assume that hypergraphs do not have isolated vertices --i.e. every 
vertex is in some edge. 

We may define from a hypergraph a finite implication algebra $\mathcal{I}_{\rsf{H}}$ 
as $\Set{X | X\subseteq e\text{ for some }e\in E}$ ordered by reverse inclusion. 

Conversely, given a finite implication algebra $\mathcal{I}$ we first 
compute the enveloping algebra $B$ and let 
$H$ be the set of coatoms of $B$ and $e\in E$ iff there is a minimal 
element $e'$ of $\mathcal{I}$ such that $e=\Set{h\in H | e'\le h}$. 
Then $\rsf{H}_{\mathcal{I}}=\brk<H, E>$. 

We would like to observe the relationship between these two 
constructions. 

\begin{lem}\label{lem:envCoat}
    Let $\mathcal{I}$ be a finite implication algebra and 
    $B=\env(\mathcal{I})$ be its enveloping Boolean algebra. Then 
    the coatoms of $B$ are exactly the coatoms of $\mathcal{I}$. 
\end{lem}
\begin{proof}
    As $\mathcal{I}$ is upwards-closed in $B$ we know that any 
    $\mathcal{I}$-coatom is also a $B$-coatom. 
    
    Let $\Set{a_{1}, \dots, a_{k}}$ be all the minimal elements of 
    $\mathcal{I}$. Then we have $\bigwedge_{i=1}^{k}a_{i}=\zero$ (by 
    minimality of the envelope). 
    Let $c$ be any $B$-coatom. Then 
    \begin{align*}
	c & =c\join \zero  \\
	 & = c\join\bigwedge_{i=1}^{k}a_{i}  \\
	 & =\bigwedge_{i=1}^{k}(c\join a_{i}). 
    \end{align*}
    Now $c\join a_{I}$ is either $\one$ or $c$. If every $c\join 
    a_{i}=\one$ then $c=\one$ -- contradiction. Hence $c\join a_{i}=c$ 
    for some $I$ and so $a_{i}\le c$. Thus $c\in\mathcal{I}$ and so 
    must be an $\mathcal{I}$-coatom. 
\end{proof}

\begin{thm}\label{thm:hypIso}
    Let Let $\mathcal{I}_{1}$ and $\mathcal{I}_{2}$ be two finite 
    implication algebras. Then 
    \begin{enumerate}[(a)]
	\item  $\mathcal{I}_{1}\simeq\mathcal{I}_{2}$  iff 
	$\rsf{H}_{1}\simeq \rsf{H}_{2}$;  
    
	\item  For any hypergraph $\rsf{H}$ 
	$\rsf{H}_{\mathcal{I}_{\rsf{H}}}\simeq \rsf{H}$; 
	
	\item For any implication algebra $\mathcal{I}$ 
	$\mathcal{I}_{\rsf{H}_{\mathcal{I}}}\simeq \mathcal{I}$. 
    \end{enumerate}
\end{thm}
\begin{proof}
    \begin{enumerate}[(a)]
	\item[]
	\item  It suffices to note that if $\rsf{H}_{I}$ that 
	if $X$ is a family of edges coming from the minimal elements of
	$X'$ then $\card{\bigcap X}= \cork(\bigvee X')$ and so the result 
	follows from \thmref{thm:minRk}. 
    
	\item  In $\mathcal{I}_{\rsf{H}}$ we see that coatoms 
	correspond precisely to vertices of $\rsf{H}$ and the minimal 
	elemenst with edges -- so reconstructing $\rsf{H}$ works -- 
	just define $f\colon V\to \text{CoAt}(\mathcal{I})$ in the 
	natural way and observe that $f$ preserves edges. 
    
	\item  This follows from parts (a) and (b) -- as (b) gives 
	$\rsf{H}_{\mathcal{I}_{\rsf{H}_{\mathcal{I}}}}\simeq 
	\rsf{H}_{\mathcal{I}}$,  and so (by (a)) we have 
	$\mathcal{I}_{\rsf{H}_{\mathcal{I}}}\simeq \mathcal{I}$.	
    \end{enumerate}
\end{proof}

We note the connection between implication profiles and the 
graphical polymatroid of a hypergraph. The latter is defined as 
$\rho(S)=\card{\bigcup_{e\in S}}e$ where $S$ is any subset of $E$. 

In $\mathcal I_{\rsf H}$ we have 
$p(S)=\cork(\bigvee S)= \card{\bigcap S}$. Therefore $p$ and $\rho$ 
are related by inclusion-exclusion:
$$
\rho(S)=\sum_{\emptyset\not=T\subseteq S}(-1)^{\card T+1}p(T).
$$

\section{Implication Profiles} 

From \thmref{thm:minRk} we know that an implication profile is 
characterizes the algebra it came from. In this section we would like 
to consider which functions can be profiles. To this end we first consider 
a slight generalization and some properties of these functions. 

\begin{defn}\label{def:iprofileII}
    Let $\mathcal{I}$ be a finite implication algebra contained in a Boolean algebra $B$ as an upper segment.
    Let $b\in B$. The \emph{profile of $\mathcal{I}$ at $b$} is the function
    $$
    p_{b, I}\colon \wp(M_{I})\setminus\Set{\emptyset}\to\N
    $$
    defined by 
    $$
    p_{b, I}(S)=k\text{ iff }[\bigvee S\join b,  1]\simeq 2^{k}.
    $$
\end{defn}

\begin{lem}\label{lem:submod}
    Let $\mathcal{I}$ be a finite implication algebra contained in a Boolean algebra $B$ as an upper segment.
    Let $b\in B$ and $p$ be the profile of $\mathcal{I}$ at $b$. Then
    \begin{enumerate}[(a)]
	\item  $p$ is decreasing; 
    
	\item  $p$ is submodular,  ie for any non-disjoint $S_{1}$ and $S_{2}$ contained in $M_{I}$ we have
	$p(S_{1})+p(S_{2})\le p(S_{1}\cup S_{2})+p(S_{1}\cap S_{2})$.
    \end{enumerate}
\end{lem}
\begin{proof}
    \begin{enumerate}
	\item  This is immediate,  as $S_{1}\subseteq S_{2}$ implies $\bigvee S_{1}\le \bigvee S_{2}$
	and so $[\bigvee S_{2}\join b, 1]$ is an upper segment of $[\bigvee S_{1}\join b, 1]$.
    
	\item  First notice that $\bigvee(S_{1}\cap S_{2})\le \bigvee S_{1}\meet \bigvee S_{2}$ and 
	$\bigvee(S_{1}\cup S_{2})=\bigvee S_{1}\meet \bigvee S_{2}$.
	Then we get 
	\begin{align*}
	    p(S_{1})+p(S_{2}) & = \text{ht}[\bigvee S_{1}, 1]+ \text{ht}[\bigvee S_{2}, 1] \\
	     & = \text{ht}[\bigvee S_{1}\join\bigvee S_{2}, 1]+\text{ht}[\bigvee S_{1}\meet \bigvee S_{2}, 1]  \\
	     & \le \text{ht}[\bigvee (S_{1}\cup S_{2}), 1]+\text{ht}[\bigvee (S_{1}\cap S_{2}), 1]  \\
	     & =p(S_{1}\cup S_{2})+p(S_{1}\cap S_{2}).
	\end{align*}
    \end{enumerate}
\end{proof}

\begin{thm}\label{thm:profile}
    Let $\mathcal{I}$ be a finite implication algebra
    and $p$ be the profile of $\mathcal{I}$. Then
    \begin{enumerate}[(a)]
	\item  $p$ is decreasing; 
    
	\item  $p$ is submodular; 
    
	\item  for any $\emptyset\not=A\subseteq M_{I}$ the functions
	$p_{A}$ and $q_{A}$ from $\wp(M_{I}\setminus A)\setminus\Set{\emptyset}\to\N$ defined by
	\begin{align*}
	    p_{A}(X) & = p(X\cup A)  \\
	    q_{A}(X) & = p(X)-p(X\cup A)
	\end{align*}
	are both decreasing and submodular.
    \end{enumerate}
\end{thm}
\begin{proof}
    As every implication algebra embeds as an upper segment of some Boolean algebra we can deduce 
    these results from the lemma above.
    
    The first two follow as $p=p_{0,  I}$ and $p_{A}=p_{\bigvee A,  I}$ gives half of the next one.
    Finally, if we let $a=\bigvee A$ then we have
    $\bigvee X=(\bigvee X\join a)\meet(\bigvee X\join\comp a)$ implies
    $\text{ht}[\bigvee X, 1]=\text{ht}[\bigvee X\join a, 1]+\text{ht}[\bigvee X\join\comp a, 1]$ so that
    $q_{A}(X)=p_{\comp a,  I}$. 
\end{proof}

We note that of $p$ is any decreasing submodular function then $p_{A}$ is always submodular and decreasing.
However that is not so for $q_{A}$. That can only happen for profiles of implication algebras.

\begin{thm}\label{thm:profileII}
    Let $M$ be a finite set,  $p\colon\wp(M)\setminus\Set{\emptyset}\to \N$ be a function
    having the properties listed in the conclusion to the last theorem. Then 
    there is a finite implication algebra $\mathcal{I}$ with
    \begin{enumerate}[(a)]
	\item  $\card{M_{I}}=\card M$ (so we will assume that $M=M_{I}$); and
    
	\item  $p$ is the implication profile of $\mathcal{I}$.
    \end{enumerate}
\end{thm}
\begin{proof}
    We work by induction in some sufficiently large finite Boolean algebra -- like 
    $2^{m}$ where $m=\sum_{i\in M}p(\Set{i})$. 
    
    If $\card M=1$ we just take this Boolean algebra, matching the one element of $M$ with $\zero$.
    
    Otherwise take one element $b$ of $M$ and find $b\in B$ with $\text{ht}[b, 1]=p(\Set{b})$.
    Now we know that $p_{\Set{b}}$ and $q_{\Set b}$ are both decreasing and submodular -- in fact it is easy to see that
    they have the same properties that $p$ has. Now apply induction,  using 
    $p_{\Set b}$ in $[b, \one]$ and $q_{\Set b}$ in $[\comp b, \one]$. This gives elements
    $m_{b}\geq b$ and $m_{\comp b}$ for $m\in M\setminus\Set b$ with 
    the correct profile on each piece.
    Defining $m$ as $m_{b}\join m_{\comp b}$ for $m\in M$ gives the desired implication algebra.
\end{proof}

%
%
%
%
%

\end{document}